\documentclass[a4paper,10pt,fleqn]{article}

\usepackage{amsfonts}
\usepackage{amssymb}

\usepackage{amsmath}

\newcommand{\R}{\mathbb R}
\newcommand{\N}{\mathbb N}

\newenvironment{proof}{\noindent {\it Proof.}} {\hspace*{\fill}$\Box$\\}

\newtheorem{proposition}{Proposition}[section]
\newtheorem{theorem}[proposition]{Theorem}
\newtheorem{corollary}[proposition]{Corollary}

\newtheorem{remark}[proposition]{Remark}
\newtheorem{example}[proposition]{Example}

\begin{document}
\noindent

\title{An Optimization Approach to 
Parameter Identification in Variational Inequalities of Second Kind - II}
 
\author{Joachim Gwinner}

\date {}

\maketitle

\abstract{This paper continues the work of \cite{Gwi18}
and is concerned with the 
inverse problem of parameter identification in variational inequalities of the second kind that does not only treat the parameter linked to a bilinear form, but importantly also the parameter  linked to a nonlinear non-smooth function. Here we specify  the abstract framework of \cite{Gwi18} and cover frictional contact problems as well as other non-smooth problems from continuum mechanics.
The optimization approach  of \cite{Gwi18} to the inverse problem using the output-least squares formulation involves the variational inequality of the second kind as constraint. Here we use regularization technics of 
nondifferentiable optimization from  \cite{FacPan}, regularize  
the nonsmooth part in the variational inequality
and arrive at an optimization problem for which the constraint variational inequality is replaced by the regularized variational equation. For
this case, the smoothness of the parameter-to-solution map is studied and convergence analysis and optimality conditions are given.
}

\begin{quote}
 2010 Mathematics Subject Classification: 
49J40, 49N45, 90C26
 \\
Keywords and phrases: Variable parameter identification, ellipticity parameter, friction parameter, trilinear form, semisublinear form, regularization, adjoint analysis, optimality condition
\end{quote}

\section{Introduction} \label{sec:1}

This paper continues the work of \cite{Gwi18}
and is concerned with the 
inverse problem of parameter identification in variational inequalities of the second kind following the terminology of \cite{Glowinski1}.
A prominent example of this class is the following {\it direct} problem:
Find the function $u$ in the standard Sobolev space 
$H^1(\Omega) = \{v \in L^2(\Omega): \nabla v  \in  (L^2(\Omega))^d \} $ on a
bounded Lipschitz domain $\Omega \subset \mathbb{R}^d \, (d=2,3)$ such that 
for any $v \in H^1(\Omega)$ there holds  

\begin{eqnarray} \nonumber
&& \int_{\Omega} e(x) ~ [\nabla u \cdot \nabla (v-u) + u (v-u)] ~dx
+ \int_{ \partial \Omega} f(s) ~ (|v|-|u|) ~ds \\[1ex]
 &\geq & \int_{\Omega} g(x) ~ (v-u) ~dx \,. \label{VI-1}
\end{eqnarray}

This variational inequality (VI) is related to the Helmholtz partial differential equation
$ - \Delta u + u = g$ rendering the coercive bilinear form
$\nabla u \cdot \nabla v + u~v $. (\ref{VI-1}) provides a simplified
scalar model of the Tresca frictional contact problem 
in linear elasticity, as detailed later in the text.
By the classic theory of variational inequalities \cite{KiSt}
there is a unique solution $u$ of (\ref{VI-1}) if 
the datum $g$ that enters the right-hand side is given in $L^2(\Omega)$
and moreover, the ``ellipticity'' parameter $e > 0$ in $L^\infty(\Omega)$ and 
the ``friction'' parameter $f> 0$ in $L^\infty(\partial \Omega)$ are known.
Here we study the {\it inverse} problem that asks for the distributed parameters $e$ and $f$,
when the state $u$ or, what is more realistic, some approximation
$\tilde u$ from measurement is known.

We loosely follow the optimization approach of \cite{Gwi18} to this inverse problem.
This approach uses an output-least squares formulation that involves the variational inequality of the second kind as constraint.
As already shown in \cite{Gwi18}, the parameter-to-solution map is Lipschitz,
however not differentiable in general. 

The main objective of this contribution is the derivation of optimality
conditions of first order. To this purpose  we adopt a regularization approach that is similar to the regularization approach to optimal control of elliptic variational inequalities that goes back to the work of V. Barbu \cite{Bar81,Bar84}.  
 Here we apply well-known technics in the study of numerical methods for the solution of finite-dimensional variational inequalities, see \cite{FacPan},
and use smoothing functions to the plus function and to the related modulus function. Thus we can regularize  
the nonsmooth part in the variational inequality
and obtain an optimization problem for which the constraint variational inequality is replaced by the regularized variational equation. For
this case, the smoothness of the parameter-to-solution map is studied and convergence analysis and optimality conditions are given.

For a review of literature on inverse problems in VI up to 2017 we refer 
to \cite{Gwi18}.
More recent works are  \cite{KhaMig19,MigKha19}.

So in contrast to the work cited above, this contribution  on the inverse problem of parameter identification in variational inequalities  does not only treat the parameter $e$ linked to a bilinear form, but importantly also the parameter  $f$ linked to a nonlinear nonsmooth function, like the modulus function above.

This paper is organized as follows. The next section 2 collects some more variational inequalities of the second kind  which model frictional contact and are drawn from other non-smooth problems in continuum mechanics. Here we also give  an abstract framework for parameter identification of the ellipticity parameter linked to the bilinear form and of an friction parameter linked to a non-smooth function. 
Then in section 3 we present an regularization procedure smoothing the modulus function. By this approach we arrive at  a regularized VI equivalent to a variational equation for which we can prove  the differentiability of the 
parameter-to-solution map. Moreover we provide an estimate
of the error between the solution of the original VI  and the solution 
of the regularized VI. Section 4 develops the least squares optimization approach to the parameter identification problems where the original VI, respectively the regularized VI appear as constraint. Beyond existence results for these optimization problems we  establish an approximation result for the 
optimal solutions associated to the regularized VI to an optimal solution associated to the original VI when the regularization parameter goes to zero.
Finally we present optimality conditions of first order for the regularized optimization problem with the regularized VI as constraint.
 The paper ends with a  short outlook to related parameter identification problems and their solution.

\section{ Some variational inequalities and an
 abstract framework for parameter identification}\label{sec:2}

Let $\Omega$ be a bounded domain $ \subset \mathbb{R}^ d \,(d=2,3)$
 with Lipschitz boundary $\Gamma$
and nonempty boundary part $ \Gamma_C \subset \subset \Gamma$.
 Further, let $ 0 < e \in L^{\infty} (\Omega),~ 0 < f \in L^{\infty} (\Gamma_C), ~g \in L^2 (\Omega),$
 $ V=\{ v \in H^1 (\Omega): v| \Gamma \backslash \Gamma_C ~ = ~0 \}. $
Then one may consider the VI: 
Find $ u \in V $~such that for all $ v \in V$,
\begin{equation} \label{VI-2} 
 \int\limits_{\Omega} e~ \nabla u \cdot \nabla (v-u)
 + \int\limits_{\Gamma_C} f (|v|-|u|) 
\geq \int\limits_{\Omega} g(v-u)  \,.
\end{equation}  
This scalar VI is more related than (\ref{VI-1}) to the Tresca frictional contact problem
which reads as follows:
Find  $ u \in V:=\{ v \in  H^1 (\Omega,~\mathbb{R}^d): 
v| \Gamma \backslash \Gamma_C ~ = ~0 \}$
($d = 2, 3 $)
 ~ such that for all $v \in V$,
 \begin{equation} \label{VI-3} 
 \int\limits_{\Omega} [E~\sigma(u)~:~\sigma(v-u)~
+~\int\limits_{\Gamma_C} f~ (|v \cdot n|-|u \cdot n|)
 \geq\int\limits_{\Omega} g \cdot (v-u)  \,, 
\end{equation}  
where now $ E \in L^{\infty} (\Omega,~\mathbb{R}^{d \times d}_{\mbox{\tiny symm}})~,
~~E>0$ (that is, $E$ is positive  definite) is the anisotropic elasticity tensor,
 $\sigma=\sigma(u), \sigma = 1/2 ~ ( \nabla u + (\nabla u)^T) $ denotes the strain field  associated to the displacement field $u$, $n$ stands for the outward unit normal, and now $g \in L^2 (\Omega,~\mathbb{R}^d ).$

A  vectorial VI of second kind similar to (\ref{VI-3}) appears in Stokes flow
with leaky boundary condition or with Tresca friction boundary condition; see 
\cite{Fujita94,BaffSass15,ABGS16}.

When replacing the functional $\int\limits_{\Gamma_C}  f~ |w| $
by $\int\limits_{\Omega}  f~ |\nabla w| $ in (\ref{VI-1})
or in (\ref{VI-2}) one obtains a VI of second kind that models
laminar flow of a Bingham fluid, see e.g. \cite{Glowinski1}.  More general than Bingham fluid
is the vectorial viscoplastic fluid flow problem studied in \cite{GloWac11}.

All these VIs can be covered by the abstract framework of
\cite{Gwi18} as follows.
Let (as above) $V$ be a real Hilbert space, where in virtue of the Riesz isomorphism we identify the dual space $V^*$ with $V$. 
Moreover let $E,F$ real Banach spaces
with convex closed cones  $E_{+} \subset E$ and $ F_{+} \subset F$.
Let as with \cite{GocKha07}, $t:E \times V \times V \rightarrow \mathbb{R},
(e,u,v) \mapsto t(e,u,v)$ a trilinear form and $l:V \rightarrow \mathbb{R},
v \mapsto l(v)$ a linear form.  Assume that $t$ is continuous
such that $t(e, \cdot , \cdot) $ is  $V-$elliptic for any fixed 
$e \in \mbox{ int } E_{+} $.
Now in addition we have a  "semisublinear form"
 $s:F \times V  \rightarrow \mathbb{R},
(f,u) \mapsto s(f,u)$, that is, for any $u \in V$,
$s(f,u)$ is linear in its first argument $f$ on F
 and for any  $f  \in F_{+}$, $s(f, \cdot) $ is  sublinear, 
continuous, and nonnegative on $V$.
Moreover assume that $s(f, 0_V) = 0 $ for any $f \in F$.

Then the forward problem is the following VI: 
Given $e \in \mbox{ int } E_{+} $ and
$f \in F_{+} $,
find  $u \in V$  such that 
\begin{equation} \label{VI} 
t(e; u,v-u)~+~ s(f;v)~-~s(f;u) ~\geq~ l(v-u)  \,,\, \forall v \in V \,.
\end{equation}  
 
Now with given convex closed subsets
 $E^{\mbox{\tiny ad}} \subset \mbox{ int } E_{+}$ and
$F^{\mbox{\tiny ad}} \subset F_{+}$
we seek to identify two parameters, 
namely the "ellipticity" parameter
$e$ in $E^{\mbox{\tiny ad}}$ and
 the "friction" parameter
$f$ in $F^{\mbox{\tiny ad}}$.

In the model problem we have some convex closed cone
$ E_{+}  \subseteq  \{ e \in L^{\infty} (\Omega):
 ~e \geq 0~~\mbox{a.e. on }  \Omega \} $
containing the convex closed "feasible" set 
$E^{\mbox{\tiny ad}} = \{ e \in E_{+} ~:~~\underline{e} \leq e \leq \overline{e}
 ~\mbox{ a.e.~on } \Omega \} $, where the bounds
 $ \underline{e} < \overline{e}$ are given in 
$\mathbb{R}_{++} = \{r \in \mathbb{R}: r > 0 \} $.
Likewise we have some convex closed cone
$ F_{+} \subseteq \{f \in L^{\infty} (\Gamma_C) ~:~~ f \geq 0 ~
\mbox {a.e. on } \Gamma_C \} $
containing the convex closed "feasible" set 
$F^{\mbox{\tiny ad}} = \{ f \in F_{+}:~\underline{f} \leq f \leq \overline{f} 
~~\mbox{a.e.~on } \Gamma_C \} $, where the bounds
$ 0 \leq \underline{f} < \overline{f}$ are given.

As with M.~S.~Gockenbach and A.~A.~Khan in \cite{GocKha07} we can assume
that $t(e;\cdot,\cdot)$ is symmetric and
\begin{eqnarray}
&& t(e; u,v) \leq \overline{t}~
 ||e||_E ~ ||u||_V ~ ||v||_V , \, \forall  e \in E,
~u \in V,~v \in V   \label{tri-est-1}
 \\[1.0ex] \label{tri-est-2}
&& t(e; u,u) \geq \underline{t}~ ||u||_V^ 2, \,
\forall e \in E^{\mbox{\tiny ad}}  \subset E,~u \in V \,. 
\end{eqnarray}
In fact, $\overline{t} < \infty$ directly follows from the 
assumed continuity of $t$, where in the model problem $\overline{t}=1$,
and $\underline{t} >0 $ comes from Poincar\'{e} inequality; 
respectively in the elastic friction contact problem 
from Korn's inequality.

In the following we specify the abstract approach of \cite{Gwi18}
and give the sublinear functional $s$ a more concrete form.
Let $D$ be some finite dimensional bounded open domain.
We fix $F :=L^{\infty}(D)$ and $F_+ :=L_{+}^{\infty} (D)$. 
Introduce the $L^2$ scalar product $(\cdot,\cdot)$ on $D$
and let $\gamma: V \to L^2(D)$ a
linear continuous map. Then use the modulus function $|\cdot|$ 
and define
\begin{equation} \label{sub-mod} 
s(f;v):= (f,|\gamma(v)|) = \int\limits_D f~ |\gamma(v)| ~ ds , ~\, 
\forall f \in L^{\infty}(D), v \in V \,.
\end{equation}  

In the model problem we have the trace map 
$\gamma: V = H^1(\Omega) \to L^2(D)$ with $D = \partial \Omega $,
whereas in the elastic friction contact problem we have
the trace map 
$\gamma: V=\{ v \in  H^1 (\Omega,~\mathbb{R}^d): 
v| \Gamma \backslash \Gamma_C ~ = ~0 \} \to L^2(D)$ with 
$D = \Gamma_c $ defined by $v \to v \cdot n$.
In the Bingham flow problem or (simplified) plasticity problem,
simply $D = \Omega$, $\gamma: v\in V \mapsto \nabla v \in L^2(\Omega)$, and the modulus function changes to the modulus of a vector in an obvious way. 

\section{ The regularization procedure}\label{sec:3}

 In this section we show how the VI (\ref{VI}) of the second kind that appears as constraint in the parameter identification problem regularizes in a variational equation. This regularization procedure is based on an appropriate smoothing of the modulus function.

\subsection{Smoothing the modulus function}

The modulus function $m, m(t)= |t|$ is related to the plus function 
$p, p(t)= t_+ = \max(t,0)$, ($t \in \R$) by
\begin{eqnarray*}
m(t) = p(t) + p(-t), \\
p(t) = 1/2 (t + m(t)) \,.
\end{eqnarray*}
Therefore smoothing functions for the plus function easily translate to those
for the modulus function. 

The general methods for constructing smoothing functions go back to the works of Sobolev, see \cite{AdaFou}, and are based on convolution. 
Let $\rho : \R \to \R_+ := \{ s \in \R: s \ge 0 \}$ be the density of a probability distribution on $\R$, that is, 
\begin{itemize}
\item $\rho$ is Lebesgue integrable
\item $ \int_\R \rho(s) ds  = 1$ \,.
\end{itemize}

Let $\R_{++} := \{ \varepsilon \in \R: \varepsilon > 0 \}$.
Then define the smoothing function  $ P : \R_{++} \times \R\to \R$   via convolution for the  plus function $p$  by
\begin{equation} \label{regP}
P(\varepsilon,t)= \int_\R p(t- \varepsilon s) \rho \, ds =
\int_{-\infty}^{\frac{t}{\varepsilon}} 
(t-\varepsilon s)\rho(s) \, ds. 
\end{equation}

Further, we focus $\rho : \R \to \R _{+}$ to be a density function of finite absolute mean, that is,  
\begin{equation} \label{fin-abs-mean}
  k:=\int_{\R}|s|\,\rho(s) \, ds < \infty.
\end{equation}

Referring to \cite[Proposition 11.8.10]{FacPan} we have the following approximation result.

\begin{proposition} \label{p-reg-prop}
Let  the density function $\rho : \R \to \R _{+}$ satisfy (\ref{fin-abs-mean}).
Then for any $\varepsilon > 0, \, t \in \R$ 
\begin{equation} \label{p-reg-est}
  | P(\varepsilon,t) - p(t) | \le k ~ \varepsilon \,.
\end{equation}
Further for any $\varepsilon > 0$, $ P(\varepsilon,\cdot)$ is convex
and twice continuously differentiable on $\R$ with
$$
P_t(\varepsilon,t) = \int_{-\infty}^{\frac{t}{\varepsilon}} \rho(s) \,  ds, \,
 P_{tt}(\varepsilon,t) =
 \frac{1}{\varepsilon}\, \rho(\frac{t}{\varepsilon})  ,
$$
and $0 \le P_t(\varepsilon,t) \le 1$ for all $t \in \R$.
\end{proposition}

By the above relation between the plus function and the modulus function, 
we  define  the smoothing function
 $ M : \R_{++} \times \R\to \R$  for the  
modulus function $m$  by
\begin{equation} \label{regM}
M(\varepsilon,t) = P(\varepsilon,t) + P(\varepsilon,-t) 
\end{equation}
and obtain from Proposition \ref{p-reg-prop} the direct consequence:

\begin{corollary} \label{p-reg-cor}
Let  the density function $\rho : \R \to \R _{+}$ satisfy (\ref{fin-abs-mean}).
Then for any $\varepsilon > 0, \, t \in \R$ 
\begin{equation} \label{m-reg-est}
  | M(\varepsilon,t) - m(t) | \le 2 k  \varepsilon \,.
\end{equation}
Further for any $\varepsilon > 0$, $ M(\varepsilon,\cdot)$ is convex
and twice continuously differentiable on $\R$ with
\begin{equation} \label{m-C-2}
M_t(\varepsilon,t) 
= \int_{-\frac{t}{\varepsilon}}^{\frac{t}{\varepsilon}} \rho(s) ds , \,
 M_{tt}(\varepsilon,t) =
 \frac{1}{\varepsilon} [ \rho(\frac{t}{\varepsilon}) + \rho(\frac{-t}{\varepsilon}] \,,
\end{equation}
and $0 \le M_t(\varepsilon,t) \le 2$ for all $t \in \R$.
\end{corollary}

Since the smoothing function for the modulus function is  based 
on the smoothing function for the plus function, 
some examples from \cite{FacPan} and the references therein are in order.
\begin{example}
$$
P_1(\varepsilon,t)= \int _{-\infty}^{\frac{t}{\varepsilon}}(t-\varepsilon s)
\,\rho_1(s) \, ds= t + \varepsilon \, ln (1+ e^{- \frac{t}{\varepsilon}})=
\varepsilon \, ln (1+ e^{ \frac{t}{\varepsilon}})\, , 
$$
where $\rho_1(s)=\frac{e^{-s}}{(1+e^{-s})^2} $. 
\end{example}
\begin{example}
$$
 P_2(\varepsilon,t)=  \int _{-\infty}^{\frac{t}{\varepsilon}}(t-\varepsilon s)
\,\rho_2(s) \, ds =\frac{\sqrt{t^2+4\varepsilon^2}+t}{2}\, ,
$$
where $  \rho_2(s)=\frac {2}{(s^2+4)^{3/2}}$.
\end{example}
\begin{example}
$$
P_3(\varepsilon, t) = \int _{-\infty}^{\frac{t}{\varepsilon}}(t-\varepsilon s)
\,\rho_3(s) \, ds =\left \{ \begin{array} {ll} 0, & \quad \mbox{if} \; \; t < -\frac
    {\varepsilon}{2},\\[0.1cm]
\frac{1}{2\varepsilon}(t+ \frac{\varepsilon}{2})^2, & \quad \mbox{if} \;\;  -
\frac{\varepsilon}{2} \leq t \leq \frac{\varepsilon}{2}\, , \\[0.1cm]
t,  & \quad \mbox{if} \;\; t > \frac    {\varepsilon}{2},
\end{array}
\right. 
$$
where $\rho_3(s)=\left \{ \begin{array} {ll} 1, & \quad \mbox{if} \;  -\frac{1}{2}\leq s \leq \frac{1}{2},
     \\[0.1cm]
0, & \quad \mbox{otherwise}.
\end{array} \right.$
\end{example}
\begin{example}
$$
P_4(\varepsilon, t) = \int _{-\infty}^{\frac{t}{\varepsilon}}(t-\varepsilon s)
\,\rho_4(s) \, ds=\left \{ \begin{array} {ll} 0, & \quad \mbox{if} \; \; t < 0,\\[0.1cm]
\frac{t^2}{2\varepsilon},& \quad \mbox{if} \; \; 0 \leq t \leq \varepsilon,\\[0.1cm]
t - \frac{\varepsilon}{2}, & \quad \mbox{if} \; \; t > \varepsilon\, ,
\end{array}
\right. 
$$
where $\rho_4(s)= \left \{ \begin{array} {ll} 1, & \quad \mbox{if} \; \;  0\leq s \leq 1,
     \\[0.1cm]
0, & \quad \; \mbox{otherwise}.
\end{array} \right.$
\end{example}

\begin{remark} 
The regularization used in the analysis of state constrained semilinear elliptic 
VIs in \cite[section 5.2]{Tiba90} is similar, but different. Indeed, here 
we regularize the special convex nonsmooth modulus function which has the
maximal monotone graph  
\[ \beta(t) := \partial m(t) =
\left\{    \begin{array}{r@{\quad:\quad}l}
                         -1 ~ & t < 0  \\
                    \mbox{\rm [} -1,1 \mbox{\rm] } & t = 0 \\
                         1 ~  & t > 0 \,.
                        \end{array}
                      \right.
\]
So $\beta$ is single-valued a.e. and \cite[p. 113]{Tiba90} regularizes 
$\beta$ to 
\begin{eqnarray*}
&&  \beta_{\varepsilon} (t)  
=\int_{-\infty}^{\infty} \beta(t + \varepsilon ~s)  \rho(s) \,  ds  \\ &&
=   - \int_{-\infty}^{ - \frac{t}{\varepsilon} } \rho(s) \,  ds  ~+~
\int_{-\frac{t}{\varepsilon}}^{\infty} \rho(s) \,  ds \,,
\end{eqnarray*}
what differs from the corresponding
$M_t(\varepsilon,t)$ in (\ref{m-C-2}) - apart from the regularizing kernel
in \cite{Tiba90} assumed to be $C^{\infty}$ with support in $[0,1]$, whereas we require that  the density function $\rho$ has finite absolute mean.
 \end{remark}

\subsection{Regularizing the VI of second kind}

 Using the smoothing function $M_\varepsilon := M(\varepsilon, \cdot)$ to the 
modulus function $m =|\cdot|$, we define the smoothing approximation to $s$ by  
\begin{equation} \label{sub-mod-reg} 
S_\varepsilon(f,v):= \int\limits_D f~ M_\varepsilon (\gamma(v)) ~ ds , ~\, 
\forall f \in L^{\infty}(D), v \in V , \varepsilon \in \R_{++}\,.
\end{equation}  

Thus with $ \varepsilon \in \R_{++} $ fixed, we can now regularize the above VI of second kind (\ref{VI}) by the following VI: 
Given $e \in \mbox{ int } E_{+} $ and
$f \in F_{+} $,
find  $u_\varepsilon \in V$  such that 
\begin{equation} \label{VI-reg} 
t(e; u_\varepsilon , v-u_\varepsilon)~+~ S_\varepsilon(f,v)~-~S_\varepsilon(f;u_\varepsilon) ~\geq~ l(v-u_\varepsilon)  \,,\, \forall v \in V \,.
\end{equation}

Since $M'_\varepsilon$ is bounded,
 we obtain in virtue of the Lebesgue theorem of majorized convergence 
\begin{eqnarray*}
 \frac{d}{dt} S_\varepsilon(f, u + t w ) |_{t=0} 
&=& \int_D f ~ M'_\varepsilon (\gamma u) \gamma(w)  ~ ds \\ 
&= &
( f ~ M'_\varepsilon(\gamma u), \gamma(w) ) \\
&= & \langle \gamma^* (f ~ M'_\varepsilon(\gamma u)), w \rangle_{V \times V} \,,
\end{eqnarray*}
hence the partial derivative of $S_\varepsilon$ with respect to $v$:
\begin{equation} \label{Deriv} 
D_v S_\varepsilon (f,u) = \gamma^* (f ~ M'_\varepsilon(\gamma u)) \in V \,.
\end{equation}

Further $S_\varepsilon(f, \cdot)$ is convex.
Therefore the VI (\ref{VI-reg})  
is indeed equivalent to the {\it variational equation}:
Find  $u_\varepsilon \in V$  such that   
\begin{equation} \label{VE-reg} 
T(e)  u_\varepsilon ~+~ D_v S_\varepsilon(f,u_\varepsilon) =  l \,,
\end{equation}
where the parameter dependent linear operator $T(e)$ is defined by
$$ 
\langle T(e) v , w \rangle  = t(e; v, w) \,.
$$
 
By the classic theory of variational inequalities \cite{KiSt}
there is a unique solution $u$ of (\ref{VI})
 and a unique solution $u_\varepsilon$ of (\ref{VI-reg}) 
for given "ellipticity" parameter
$e$ in $E^{\mbox{\tiny ad}}$ and
  "friction" parameter
$f$ in $F^{\mbox{\tiny ad}}$.
This leads to uniquely defined solution maps 
$(e,f) \in E^{\mbox{\tiny ad}} \times F^{\mbox{\tiny ad}}
\mapsto u = {\cal S}(e,f)$,
$(e,f) \in E^{\mbox{\tiny ad}} \times F^{\mbox{\tiny ad}}
\mapsto u_\varepsilon = {\cal S}_\varepsilon(e,f)$, respectively.
In \cite{Gwi18} we have shown that the solution map ${\cal S}$
is Lipschitz both in 
ellipticity and in friction parameter. However, the parameter-to-solution map ${\cal S}$ is not smooth in general. 

An advantage of replacing the variational inequality by the regularized  equation is that for the latter the parameter-to-solution map
 ${\cal S}_\varepsilon$  is {\it smooth} as shown below.

\begin{theorem} \label{sol-map-smooth} 
Let $\varepsilon >0$.
\begin{enumerate}
\item 
Fix $f \in F^{\mbox{\tiny ad}}$. 
Then the map $e \mapsto v_\varepsilon(e) :=
{\cal S}_\varepsilon(e,f)$ is differentiable at any point $e$ in the interior of $E^{\mbox{\tiny ad}}.$ 
For any direction $\delta e\in \hat E$, the derivative 
$\delta v_\varepsilon= Dv_\varepsilon (e)(\delta e)$ is the unique solution of
the following equation:
\begin{equation}\label{DFVE-1}
T(e) \delta v_\varepsilon + 
\gamma^*(f M_\varepsilon^{''}(\gamma v_\varepsilon)(\gamma \delta v_\varepsilon))
=-T(\delta e) v_\varepsilon 
\end{equation}
\item 
Fix $e \in E^{\mbox{\tiny ad}}$. 
Then the map $f \mapsto w_\varepsilon(f) :=
{\cal S}_\varepsilon(e,f)$ is differentiable at any point $f$ in the interior of $F^{\mbox{\tiny ad}}.$ 
For any direction $\delta f\in \hat F$, the derivative 
$\delta w_\varepsilon= Dw_\varepsilon (f)(\delta f)$ is the unique solution of
the following equation:
\begin{equation}\label{DFVE-2}
T(e) \delta w_\varepsilon + 
\gamma^*(f M_\varepsilon^{''}(\gamma w_\varepsilon)(\gamma \delta w_\varepsilon))
=- \gamma^*(\delta f M'_\varepsilon (\gamma w_\varepsilon)) 
\end{equation}
\item
The map $(e,f) \in E^{\mbox{\tiny ad}} \times F^{\mbox{\tiny ad}}
\mapsto u_\varepsilon = {\cal S}_\varepsilon(e,f)$
is differentiable at any point $(e,f)$ in the interior of 
$E^{\mbox{\tiny ad}} \times F^{\mbox{\tiny ad}}.$ 
From (\ref{DFVE-1}),(\ref{DFVE-2}) one obtains the derivative 
$\delta u_\varepsilon= Du_\varepsilon (e,f)(\delta e,\delta f)
 =(\delta v_\varepsilon,\delta w_\varepsilon)$ for any direction 
 $(\delta e,\delta f) \in \hat E \times \hat F$.
\end{enumerate}
\end{theorem}

\begin{proof} 
\begin{enumerate}
\item 
The differentiability  of the map
 $e \mapsto v_\varepsilon(e)=
{\cal S}_\varepsilon(e,f)$
follows by applying the implicit function theorem to
the map $G:\hat E\times V\to V$ mapping $(e,v)\mapsto 
T(e) v + \gamma^* (f M'_\varepsilon (\gamma v)) $.
 The derivative
$D_{v}G(e,v): V\rightarrow V$
 is given by 
$D_{v}G(e,v)(\delta v) =
T(e) \delta v +
 \gamma^* (f M^{''}_\varepsilon (\gamma v)(\gamma \delta v)) $. 
For every $l\in V$, the variational equation
\begin{equation*}
t(e;\delta v,w) + 
(f M^{''}_\varepsilon (\gamma v)(\gamma \delta v),\gamma w)_{L^2(D) \times L^2(D)} 
=\left\langle l,w \right\rangle_{V\times V}, \forall w \in V
\end{equation*}
possesses a unique solution $\delta v$,
since $t$ is uniformly coercive by \eqref{tri-est-2}
and  $M^{''}_\varepsilon (\gamma v) \ge 0$ 
by \eqref{m-C-2} and $\rho \ge 0$ . 
Therefore $D_{v}G(e,\cdot)(v):V\rightarrow V$ is surjective and the differentiability follows from the implicit function theorem.
From  \eqref{VE-reg}  we have
\begin{equation*}
T(\delta e) v_\varepsilon + T(e) \delta v_\varepsilon +
\gamma^* (f M^{''}_\varepsilon (\gamma v_\varepsilon)
(\gamma \delta v_\varepsilon)) =0,
\end{equation*}
 and \eqref{DFVE-1} follows.
\item
Similarly the differentiability  of the map
 $f \mapsto w_\varepsilon(f)=
{\cal S}_\varepsilon(e,f)$
follows by applying the implicit function theorem to
the map $H:\hat F\times V\to V$ mapping $(f,w)\mapsto 
T(e) w + \gamma^* (f M'_\varepsilon (\gamma w)) $.
 The derivative
$D_{w} H(f,w):  V\rightarrow V$  is given by 
$D_{w} H(f,w)(\delta w) =
T(e) \delta w +
 \gamma^* (f M^{''}_\varepsilon (\gamma w)(\gamma \delta w)) $. 
Similarly as above  $D_{w} H(f,\cdot)(w):V\rightarrow V$ is 
seen to be surjective and the differentiability follows from the implicit function theorem. Again from  \eqref{VE-reg} we have
\begin{equation*}
T(e) \delta w_\varepsilon +
 \gamma^*(\delta f M'_\varepsilon (\gamma w_\varepsilon)) +
\gamma^* (f M^{''}_\varepsilon (\gamma w_\varepsilon)
(\gamma \delta w_\varepsilon)) =0,
\end{equation*}
 and \eqref{DFVE-2} follows.
\item
Clear from above. 
\end{enumerate}
The proof is complete.
\end{proof}

\begin{remark} 
The assumption on interior points in Theorem \ref{sol-map-smooth}
does not pose a restriction of generality.
Indeed, some feasible $f \in F^{\mbox{\tiny ad}} = 
\{ f \in L^{infty}_{+} (D):~\underline{f} \leq f \leq \overline{f} 
~~\mbox{a.e.~on } D \}  =:F_{\underline{f}}^{\overline{f}} $
is an interior point of $F_{\underline{f}/2}^{2 \overline{f}}$
and so apply the above result in the latter larger set, simlarly 
for  $e \in E^{\mbox{\tiny ad}}$. 
\end{remark}

\subsection{An estimate of the regularization error}

To conclude this section we give an estimate of the error between
the solution $u = {\cal S}(e,f)$ of the original VI (\ref{VI})
and the solution $u_\varepsilon = {\cal S}_\varepsilon(e,f)$
of the regularized VI (\ref{VI-reg}).

\begin{theorem} \label{theo-reg-error}
Let  the density function $\rho : \R \to \R _{+}$ satisfy (\ref{fin-abs-mean}).
Let $e \in E^{\mbox{\tiny ad}}$ and $f  \in F^{\mbox{\tiny ad}}$.
Then for the solution $u = {\cal S}(e,f)$ of the original VI (\ref{VI})
and the solution $u_\varepsilon = {\cal S}_\varepsilon(e,f)$
of the regularized VI (\ref{VI-reg}) there holds the error estimate
 \begin{equation}\label{reg-error-est}
	\| S_\varepsilon (e,f) - S(e,f) \|_V  = 
{\cal O}(\varepsilon^{\frac{1}{2}})  \,. 
 \end{equation}
\end{theorem}

\begin{proof}
Inserting $v=u_\varepsilon$ in the original VI (\ref{VI}) and $v=u$ in the regularized VI (\ref{VI-reg}), gives
\begin{align*}
t(e;u,u_\varepsilon - u) + \int_D f ~[m(\gamma(u_\varepsilon)) - m(\gamma(u))] ~ds
&\geq l(u_\varepsilon - u) \,,\\
t(e;u_\varepsilon,u-u_\varepsilon) + \int_D f ~ 
[M_\varepsilon(\gamma(u)) - M_\varepsilon(\gamma(u_\varepsilon))] ~ds
&\geq l(u - u_\varepsilon) \,. 
\end{align*}
A rearrangement of the above yields, where we use 
 (\ref{tri-est-2}) and Corollary \ref{p-reg-cor}, (\ref{m-reg-est}),    
\begin{eqnarray*}
&& \underline{t} \|u_\varepsilon - u \|_V^2 \le
t(e;u_\varepsilon - u, u_\varepsilon - u)  \\   
&& \le  
\int_D f ~ [m(\gamma(u_\varepsilon)) - M_\varepsilon(\gamma(u_\varepsilon))
+ M_\varepsilon(\gamma(u))   - m(\gamma(u)) ] ~ds \\
&& \le 4 c ~ \varepsilon ~ |D| \|f\|_{L^\infty(D)}  \,,
\end{eqnarray*}
where $|D|$ is the Lebesgue measure of $D$. This proves the claimed error estimate.
\end{proof}

\section{ The optimization approach }\label{sec:4}

Let an observation $\grave{v} \in V$ be  given.
Then the parameter identification problem studied in this paper reads:
Find parameters  
$e \in E^{\mbox{\tiny ad}}, f \in F^{\mbox{\tiny ad}}$
such that $ u = {\cal S}(e,f)$ "fits best" $\grave v$, 
where $u \in V$ satisfies the VI (\ref{VI}), namely 
$$t(e;u,v-u)  + s(f;v) - s(f;u) ~\geq ~ l(v-u), \forall v \in V \,.$$
Similar to \cite{Hin01,Gon06,GocKha07} and similar to parameter estimation in linear elliptic equations \cite{BanKun}
we follow an optimization approach and introduce the 
"misfit function"
  $$J(e,f) := \frac{1}{2}~ \| {\cal S}(e,f) - \grave{v} \|^2  $$
to be minimized.

Here we assume similar to \cite{Hin01} that the sought 
ellipticity and friction parameters
are smooth enough to satisfy with compact imbeddings
$$ E^{\mbox{\tiny ad}} \subset \hat{E}  \subset \subset E; F^{\mbox{\tiny ad}} \subset \hat{F}  \subset \subset F = L^\infty(D) \subset L^2(D) . \, $$

Some examples are in order. By the Rellich-Kondrachev Theorem \cite[Theorem 6.3]{AdaFou},
$H^1(\Omega) \subset \subset C^0_B(\Omega)$, the space of bounded, continuous functions on $\Omega$,
provided $\Omega$ satisfies the cone condition; clearly $C^0_B(\Omega) \subset L^\infty(\Omega)$.
Thus $\hat{E} = H^1(\Omega) \subset \subset L^\infty(\Omega)$. 
More general Sobolev spaces of fractional order can also be used, see e.g. \cite{AssRoe13}
treating the identification of the ellipticity parameter in linear elliptic Dirichlet problems.

For simplicity let $\hat{E}, \hat{F}$ be Hilbert spaces (or more generally reflexive Banach spaces).
Thus with given weights $\alpha > 0, \beta > 0$ we pose the stabilized optimization problem

\begin{eqnarray*}
(OP) \quad && \mbox{minimize }
J(e,f) + \frac{\alpha}{2}~ \|e\|_{\hat{E}}^2 + 
\frac{\beta}{2} \|f\|_{\hat{F}}^2 \\[1ex]
 && \mbox{subject to } e\in E^{\mbox{\tiny ad}},\, f \in F^{\mbox{\tiny ad}}
\end{eqnarray*} 

Under these assumptions we have the following solvability theorem. 

\begin{theorem} \label{exis-orig}
Suppose the above compact imbeddings.
Suppose that the trilinear form $t$ satisfies
(\ref{tri-est-1}) and (\ref{tri-est-2})
and that the semisublinear form $s$ is given by (\ref{sub-mod}).
Then $(OP)$ admits an optimal (not necessarily unique!) solution 
 $(e^*,f^*,u) \in E^{\mbox{\tiny ad}} \times F^{\mbox{\tiny ad}} \times V$, 
where $u = {\cal S}(e^*,f^*)$. i.e.
$u\in V$ solves the VI (\ref{VI}) for the optimal parameter $(e^*,f^*).$ 
 \end{theorem}
 
\begin{proof}
For details see the proof of \cite[Theorem 4.1]{Gwi18}.
\end{proof}

We now consider an  analogue of $(OP)$  where the underlying variational inequality has been replaced by a regularized variational inequality
that is equivalent to a variational equation:
\begin{eqnarray*}
(OP)_\varepsilon \quad && \mbox{minimize }
J_\varepsilon (e,f) + \frac{\alpha}{2}~ \|e\|_{\hat{E}}^2 + 
\frac{\beta}{2} \|f\|_{\hat{F}}^2 \\[1ex]
 && \mbox{subject to } e\in E^{\mbox{\tiny ad}},\, f \in F^{\mbox{\tiny ad}} \,,
\end{eqnarray*} 
where
$$
J_\varepsilon(e,f) := \frac{1}{2}~ \| {\cal S}_\varepsilon (e,f) - \grave{v} \|^2 \,. 
$$

We give the following existence and convergence result:

\begin{theorem} \label{exis-reg-conv} 
Suppose the above compact imbeddings.
Suppose that the trilinear form $t$ satisfies
(\ref{tri-est-1}) and (\ref{tri-est-2})
and that the semisublinear form $s$ is given by (\ref{sub-mod}).
Moreover, let  the density function $\rho : \R \to \R _{+}$ 
in the definition  (\ref{sub-mod-reg}) of the regularizing function $S_\varepsilon$ satisfy (\ref{fin-abs-mean}).
Then for any $\varepsilon >0,~ (OP)_\varepsilon$ has an optimal solution 
$(\bar{e}_\varepsilon, \bar{f}_\varepsilon).$ Moreover for any 
sequence $\{ \varepsilon_k \}_{k \in \N} \subset \R_{++}$ with $\varepsilon_k \to 0$ for $k \to \infty$, there is a subsequence 
$\{ (\bar{e}_n,\bar{f}_n,\bar{u}_n) \}_{n \in N}$, where 
$\bar{u}_n= {\cal S}_{\varepsilon_n} (\bar{e}_n,\bar{f}_n)$ 
for $ n \in N \subset \N$,  such that for $n \to \infty$ we have 
$(\bar{e}_n,\bar{f}_n) \to (\tilde{e},\tilde{f})$ in $E \times F$,
 $\bar{u}_n\to \tilde{u}$ in $V$ where $(\tilde{e},\tilde{f})$ is a solution of $(OP)$ and $\tilde{u}={\cal S}(\tilde{e},\tilde{f})$.
\end{theorem}
\begin{proof} For a fixed $\varepsilon >0,$ the existence of a solution $(\bar{e}_\varepsilon,\bar{f}_\varepsilon)$ follows by arguments similar to those used in the proof of Theorem~\ref{exis-orig}.
The proof of the convergence result is split into the following three steps.

{\it Step 1: An a priori estimate and convergence of 
some subsequence $ \{\bar e_n,\bar f_n \}$ to 
some $(\tilde{e},\tilde{f})$}

We observe that for any optimal solution
$(e_k,f_k) := (\bar{e}_{\varepsilon_k}, \bar{f}_{\varepsilon_k})$ of 
$(OP)_{\varepsilon_k}$ and for any optimal solution
$(e^*,f^*).$ of $(OP)$,
\begin{eqnarray*}
&&J_{\varepsilon_k} (e_k,f_k) 
+ \frac{\alpha}{2}~ \|e_k\|_{\hat{E}}^2 + \frac{\beta}{2} \|f_k\|_{\hat{F}}^2 \\
&& = \frac{1}{2}~ \| {\cal S}_{\varepsilon_k} (e_k,f_k) - \grave{v} \|^2
+ ~\frac{\alpha}{2}~ \|e_k\|_{\hat{E}}^2 + \frac{\beta}{2} \|f_k\|_{\hat{F}}^2 \\
&& \le \frac{1}{2}~ \| {\cal S}_{\varepsilon_k} (e^*,f^*) - \grave{v} \|^2
+ ~\frac{\alpha}{2}~ \|e^*\|_{\hat{E}}^2 + \frac{\beta}{2} \|f^*\|_{\hat{F}}^2
 \,.
\end{eqnarray*}
Moreover, by the error estimate (\ref{reg-error-est}),
$$
\| {\cal S}_{\varepsilon_k} (e^*,f^*) - \grave{v} \|
\le \| {\cal S} (e^*,f^*) - \grave{v} \|
 + {\cal O}(\varepsilon_k^{\frac{1}{2}}) \,.
$$
Hence, the boundedness of 
$\|e_k\|_{\hat{E}} +  \|f_k\|_{\hat{F}} $ follows.
Therefore there exists  a subsequence 
$\{ (e_n,f_n) \}_{n \in N}$ with $N \subset \N$ that converges weakly to
some  $(\tilde{e},\tilde{f})
\in  E^{\mbox{\tiny ad}} \times F^{\mbox{\tiny ad}}$ in the reflexive space
$\hat E \times \hat F$. In view of the assumed
compact imbeddings there is a further subsequence again denoted by 
$\{ (e_n,f_n) \}_{n \in N}$,  such that for $n \to \infty$,  
$(e_n,f_n) \to (\tilde{e},\tilde{f}) $ in $E \times F$ strongly.

{\it Step 2: Convergence ${\cal S}_{\varepsilon_n} (e_n,f_n)
\to {\cal S}(\tilde{e},\tilde{f})$}

Let $u_n := {\cal S}_{\varepsilon_n} (e_n,f_n), 
\tilde u := {\cal S}(\tilde{e},\tilde{f})$. Then inserting $v=u_n$ in the original VI (\ref{VI}) with parameter $(\tilde e,\tilde f)$
and $v=\tilde u$ in the regularized VI (\ref{VI-reg}) with parameter
 $(e_n,f_n)$, gives
\begin{align*}
t(\tilde e;\tilde u,u_n - \tilde u) 
+ \int_D f ~[M(\gamma(u_n)) - M(\gamma(u))] ~ds
&\geq l(u_n - \tilde u) \,,\\
t(e_n;u_n, \tilde u-u_n) + 
\int_D f_n ~ [M_{\varepsilon_n}(\gamma(\tilde u)) - M_{\varepsilon_n}(\gamma(u_n))] ~ds &\geq l(\tilde u - u_n) \,. 
\end{align*}
A rearrangement of the above yields by 
 (\ref{tri-est-2}) and  (\ref{tri-est-1}) 
  
\begin{eqnarray*}
&& \underline{t} \|u_n - \tilde u \|_V^2 \le
t(e_n;u_n - \tilde u, u_n - \tilde u)  \\   
&& \le t(\tilde e - e_n;\tilde u, u_n - \tilde u) + R_n \\
&& \le
\overline{t}~\|e_n - \tilde e \|_E ~\| \tilde u \|_V ~\|u_n - \tilde u \|_V
+ R_n \,,
\end{eqnarray*} 
where using Corollary \ref{p-reg-cor}, (\ref{m-reg-est}),
\begin{eqnarray*}
 R_n &=& s(\tilde f, u_n) - S_{\varepsilon_n}(f_n,u_n) + 
S{\varepsilon_n}(f_n,\tilde u) - s(\tilde f,\tilde u) \\
& = & \int_D (\tilde f - f_n) ~ [M(\gamma u_n) 
+ f_n [M(\gamma u_n) - M_{\varepsilon_n}(\gamma u_n)] \\
&& + ~(f_n - \tilde f) M(\gamma \tilde u) 
+ f_n [M(\gamma \tilde u) - M_{\varepsilon_n}(\gamma \tilde u)] ~ds \\
&\le &
 \|\tilde f- f_n\|_{L^\infty(D)} \|\gamma\|_{V \to L^1(D)} \|u_n\|_V
 + 4 c_0 ~ \varepsilon_n ~ |D| \|f_n\|_{L^\infty(D)} \\
&& + ~ \|\tilde f- f_n\|_{L^\infty(D)} \|\gamma\|_{V \to L^1(D)} \|\tilde u\|_V
 +  ~ 4 c_0 ~ \varepsilon_n ~ |D| \|f_n\|_{L^\infty(D)}  \,.
\end{eqnarray*}
These estimates show that 
$$
\underline{t} \|u_n - \tilde u \|_V^2 \le
c_1  \|u_n - \tilde u \|_V + c_2 
$$
holds for some constants $c_1 > 0, c_2 > 0$. 
Hence $\{u_n\}$ is bounded, say
 $\max( \|u_n\|_V, \|u_n - \tilde u \|_V ) 
\le \tilde c < \infty$. 
Thus the above estimates give
$$\underline{t} \|u_n - \tilde u \|_V^2 \le
 \overline{t}~ \tilde c ~\| \tilde u \|_V ~ \|e_n - \tilde e \|_E +R_n \,,
$$
where
$$
R_n \le 
2 \tilde c ~\|\gamma\|_{V \to L^1(D)} ~ \|\tilde f- f_n\|_{L^\infty(D)} 
+  8 c_0 ~ \varepsilon_n ~ |D| \|f_n\|_{L^\infty(D)}
$$
and so $R_n \to 0$ for $n \to \infty$.
The convergence $u_n \to \tilde u$ follows.

{\it Step 3: $(\tilde{e},\tilde{f})$ is a solution of $(OP)$}

Take $(e_0,f_0) \in  E^{\mbox{\tiny ad}} \times F^{\mbox{\tiny ad}}$ arbitrarily.
Let $u_0 := {\cal S}(e_0,f_0), \bar u_n := {\cal S}_{\varepsilon_n} (e_0,f_0)$. 
Then firstly by Theorem \ref{theo-reg-error}, $\bar u_n \to u_0$ in $V$
for $n\in N \to \infty$,
secondly $(e_0,f_0)$ is feasible for the regularized optimization problem
$(OP)_{\varepsilon_n}$. Hence 
\begin{eqnarray*}
&&J (\tilde e,\tilde f) 
+ \frac{\alpha}{2}~ \|\tilde e\|_{\hat{E}}^2 
+ \frac{\beta}{2} \|\tilde f\|_{\hat{F}}^2 \\
&& = \frac{1}{2}~ \| \tilde u - \grave{v} \|^2
+ ~\frac{\alpha}{2}~ \|\tilde e\|_{\hat{E}}^2 
+ \frac{\beta}{2} \|\tilde f\|_{\hat{F}}^2 \\
&& \le \liminf_{n \to \infty}  
\frac{1}{2}~ \| u_n - \grave{v} \|^2
+\liminf_{n \to \infty}  
 ~[ \frac{\alpha}{2}~ \|e_n\|_{\hat{E}}^2 + \frac{\beta}{2} \|f_n\|_{\hat{F}}^2] \\
&& \le \liminf_{n \to \infty}  
[ \frac{1}{2}~ \| u_n - \grave{v} \|^2
+ \frac{\alpha}{2}~ \|e_n\|_{\hat{E}}^2 + \frac{\beta}{2} \|f_n\|_{\hat{F}}^2] \\ 
&& =  \liminf_{n \to \infty}  
[ J_{\varepsilon_n} (e_n,f_n) 
+ \frac{\alpha}{2}~ \|e_n\|_{\hat{E}}^2 + \frac{\beta}{2} \|f_n\|_{\hat{F}}^2] \\
&& \le  \liminf_{n \to \infty}  
[ J_{\varepsilon_n} (e_0,f_0) 
+ \frac{\alpha}{2}~ \|e_0\|_{\hat{E}}^2 + \frac{\beta}{2} \|f_0\|_{\hat{F}}^2] \\
&& =  \liminf_{n \to \infty}  
[ \frac{1}{2}~ \| \bar u_n - \grave{v} \|^2
+ \frac{\alpha}{2}~ \|e_0\|_{\hat{E}}^2 + \frac{\beta}{2} \|f_0\|_{\hat{F}}^2] \\
&& =  \frac{1}{2}~ \| u_0 - \grave{v} \|^2
+ \frac{\alpha}{2}~ \|e_0\|_{\hat{E}}^2 + \frac{\beta}{2} \|f_0\|_{\hat{F}}^2 \\
&& = J(e_0,f_0) + \frac{\alpha}{2}~ \|e_0\|_{\hat{E}}^2 + \frac{\beta}{2} \|f_0\|_{\hat{F}}^2 
 \,,
\end{eqnarray*}
what shows the optimality of $(\tilde e,\tilde f)$.

 The proof is complete.
 
\end{proof}

\begin{remark} 
Note that the optimization problems $(OP)$ and $(OP)_\varepsilon$ typically have multiple solutions. To overcome the ill-posedness, we could use the well-known method of Browder-Tykhonov regularization, see e.g. \cite{NasLiu98}. More specifically, if we have some a priori information of a solution of these problems, then this could be incorporated in the regularized problem $(OP)_\varepsilon$ in a similar way as it was done in the work 
\cite{GJKS18} that was focused to the identification of the ellipticity
parameter $e$.
\end{remark}  

Now we apply the general theory of optimality conditions in optimal control in the differentiable case presented in \cite{NST06} and derive 
from \cite[Theorem 3.1.7]{NST06}  the following optimality condition 
(of first order) for the regularized problem $(OP)_\varepsilon$.

\begin{theorem}\label{opti-reg}
Suppose the above compact imbeddings.
Suppose that the trilinear form $t$ satisfies
(\ref{tri-est-1}) and (\ref{tri-est-2})
and that the semisublinear form $s$ is given by (\ref{sub-mod}).
Moreover, let  the density function $\rho : \R \to \R _{+}$ 
in the definition  (\ref{sub-mod-reg}) of the regularizing function $S_\varepsilon$ satisfy (\ref{fin-abs-mean}).
Then for any $\varepsilon >0$, for any optimal solution 
$(\bar{e}_\varepsilon, \bar{f}_\varepsilon)$ of  $(OP)_\varepsilon$,
there exists  $\bar{p}_{\varepsilon}\in V,$ uniformly bounded in $V,$  such that
\begin{align}
T(\bar{e}_{\varepsilon} ~\bar{p}_{\varepsilon} +
\gamma^*(f M^{''}_{\varepsilon} ( \gamma \bar{u}_{\varepsilon}) 
\gamma \bar{p}_{\varepsilon}) & = v\grave- \bar{u}_{\varepsilon} \,,
\label{OPT1} \\
\alpha \langle e - \bar{e}_{\varepsilon},\bar{e}_{\varepsilon} \rangle_{\hat E \times \hat E}
 + t(e - \bar{e}_{\varepsilon},\bar{u}_{\varepsilon},\bar{p}_{\varepsilon})  
&\geq 0 ,\ \ \forall e \in E^{\mbox{\tiny ad}} \,,  \label{OPT2} \\
 \beta \langle f - \bar{f}_{\varepsilon},\bar{f}_{\varepsilon} 
\rangle_{\hat F \times \hat F}  +  (M^{'}_{\varepsilon} ( \gamma \bar{u}_{\varepsilon}) ~\gamma \bar{p}_{\varepsilon}),
f- \bar{f}_{\varepsilon}) 
&\geq 0 ,\ \ \forall f \in F^{\mbox{\tiny ad}} \,.  \label{OPT3} 
\end{align}%
\end{theorem}
\begin{proof}
Let $\varepsilon >0$ be fixed.
Here we minimize
$$
L(u_{\varepsilon},e,f) = 1/2 ~\|u_{\varepsilon} - \grave v\|^2
+ 1/2 ~ \alpha \|e\|^2_{\hat E} + 1/2 ~ \beta \|f\|^2_{\hat F}
$$ 
and we have the constraint equation 
$$
A(u,e,f) = T(e) u + \gamma^*(f M'_\varepsilon (\gamma u)) - l = 0 \,.
$$
The partial derivative $ D_u A(u,e,f)$ is given by
$$
\partial u \mapsto T(e) \delta u +
 \gamma^*(f M^{''}_\varepsilon(\gamma u) \delta u) \,.
$$
 We compute the adjoints via
$$\langle T(e) \delta u,w \rangle_{V \times V} = 
\langle T(e) w, \delta u \rangle_{V \times V}
$$
by the symmetry of $t(e;\cdot,\cdot)$ and since $f \in L^\infty(D)$,
\begin{eqnarray*}
&& \langle \gamma^*(f M^{''}_\varepsilon ~(\gamma u) ~ \gamma ~ \delta u),w
 \rangle_{V \times V}
= (f M^{''}_\varepsilon(\gamma u) ~\gamma ~\delta u,
 \gamma w)_{L^2(D) \times L^2(D)} \\
&& = (\gamma  ~ \delta u, f M^{''}_\varepsilon(\gamma u) \gamma w)_{L^2(D) \times L^2(D)} 
= \langle  \delta u, 
\gamma^*(f M^{''}_\varepsilon(\gamma u) \gamma w)\rangle_{V \times V} \,.
\end{eqnarray*}
Hence with the partial derivative $D_u L$ we obtain 
for the optimal solution
$$ T(\bar{e}_{\varepsilon} ~\bar{p}_{\varepsilon})  +
\gamma^*(f M^{''}_{\varepsilon} ( \gamma \bar{u}_{\varepsilon}) 
\gamma \bar{p}_{\varepsilon}) + \bar{u}_{\varepsilon} - \grave v = 0_V \,,$$
what is \eqref{OPT1}. By coercivity - see the arguments in the proof of
Theorem \ref{sol-map-smooth} - this equation uniquely determines  
the adjoint variable $\bar{p}_{\varepsilon}$.

The partial derivative $ D_e A(u,e,f)$ is given by
$\partial e \mapsto T(\delta e) u $.
Then  with the partial derivative $D_e L$  and 
similar to the proof of
\cite[Theorem 3.6]{GJKS18} we obtain using the symmetry of $t(e,\cdot,\cdot)$
\eqref{OPT2}.

Finally the partial derivative $ D_f A(u,e,f)$ is given by
$$
\partial f \mapsto \gamma^*( \delta f M'_{\varepsilon} ( \gamma u)) \in V \,.
$$
Since  $P'_{\varepsilon}(t) \in [0,1]$ and hence  
$M'_{\varepsilon} ( \gamma v) \in L^\infty(D)$ for any $v \in V$,
we can compute the adjoint:
\begin{eqnarray*}
&& \langle \gamma^*(\delta f M'_\varepsilon(\gamma u)),w
 \rangle_{V \times V} \\
&& = (\delta f M'_\varepsilon(\gamma u), \gamma w)_{L^2(D) \times L^2(D)} \\
&& = (\delta f,  M'_\varepsilon(\gamma u) \gamma w)_{L^2(D) \times L^2(D)} \,.
\end{eqnarray*}
Hence with the partial derivative $D_f L$,  \eqref{OPT3} follows.

We still need to show that $\{\bar{p}_{\varepsilon}\}$ is uniformly bounded. For this, we take $v=\bar{p}_{\varepsilon}$ in the variational equation to the adjoint equation \eqref{OPT1} and by using ellipticity of $t$ and 
$ M^{''}_{\varepsilon} \ge 0$, we obtain
$$\underline{t} \left\Vert \bar{p}_{\varepsilon}\right\Vert _{V}^{2} \leq 
t(\bar{a}_{\varepsilon};\bar{p}_{\varepsilon},\bar{p}_{\varepsilon})+
\langle \gamma^*(f M^{''}_{\varepsilon} ( \gamma \bar{u}_{\varepsilon}) 
\gamma \bar{p}_{\varepsilon}),\bar{p}_{\varepsilon} \rangle
\leq C_{1}\left\Vert \bar{p}_{\varepsilon} \right\Vert_{V} \,,$$
where we also used the fact that $\{\bar{u}_{\varepsilon}\}$ is also bounded. The proof is complete.
\end{proof}

\section{ Concluding remarks -
 An outlook} \label{sec:8}
 
 The field of Inverse Problems and Parameter Identification is very vast. There are many books on this subject, see also the interesting surveys \cite{BonCon05,OniShi05}.

In this paper we have studied a parameter identification problem for 
nonlinear non-smooth problems in the setting of variational inequalities of the second kind. Thus our results are confined to the class of inverse problem where the associated direct problem is a convex variational problem.

On the other hand, there are interesting nonconvex variational problems
resulting from nonmonotone boundary conditions in contact mechanics
that describe adhesion and delamination phenomena; see e.g. 
\cite{GwOv15} for the forward problem. Here 
the identification of the nonmonotone contact laws is a challenging task.
Also in fluid mechanics nonmonotone boundary conditions are encountered, see e.g. \cite{DKM15}. Here also interesting parameter identification
problems arise. When the nonmonotone boundary conditions are of max or min or min-max type, advanced regularization techniques, see 
\cite{OvGw14,GwOv20} for the forward problem, are applicable. 

With linear partial differential equations with constant coefficients,
as with linear elasticity and Stokes flow, boundary integral methods are available. Then boundary element methods \cite{JGEPS} could be developed 
for the efficient and reliable solution of the friction parameter identification problem.
.

\bibliographystyle{amsplain}

\bibliography{Ident_VI_SecKind_BIB_II}

\end{document}